\documentclass[reqno,twoside,11pt]{amsart}

\usepackage{amsmath,amsfonts,calrsfs,fullpage,amssymb,color,verbatim,eucal,yfonts,mathrsfs}
\newtheorem{Theorem}{Theorem}[section] 
\newtheorem{Definition}[Theorem]{Definition}
\newtheorem{Proposition}[Theorem]{Proposition}
\newtheorem{Lemma}[Theorem]{Lemma}
\newtheorem{Corollary}[Theorem]{Corollary}
\newtheorem{Remark}[Theorem]{Remark}

\makeatletter
\@addtoreset{equation}{section}

\makeatother

\def\R{\mathbb R}
\def\N{\mathbb N}

\def\C{\mathbb C}

\def\C{\mathbb C}
\def\B{\mathcal B}
\def\m{\textfrak{m}}

\newcommand{\one}{1\!\!\!\;\mathrm{l}}

\newcommand\scalH[2]{\left[ #1, #2\right]_H}

\title[$BV$ functions on convex domains in Wiener spaces]{$BV$ functions on convex domains in Wiener spaces}
 
\author[A. Lunardi]{Alessandra Lunardi}
\address{
Dipartimento di Matematica e Informatica\\
Universit\`a di Parma}
\email{alessandra.lunardi@unipr.it}

\author[M. Miranda]{Michele Miranda jr.}
\address{
Dipartimento di Matematica e Informatica\\
Universit\`a di Ferrara}
\email{michele.miranda@unife.it}

\author[D. Pallara]{Diego Pallara}
\address{
Dipartimento  di Matematica e Fisica 
``E. De Giorgi"\\
Universit\`a del Salento}
\email{diego.pallara@unisalento.it}



\begin{document}

 \begin{abstract}  
We study functions of bounded variation defined in an abstract Wiener space $X$, relating the variation of a function $u$ on a convex open set $\Omega\subset X$ to the behavior near $t=0$ of $T(t)u$, $T(t)$ being the Ornstein--Uhlenbeck semigroup in $\Omega$. 
 \end{abstract}

\maketitle


\section{Introduction}
This paper is devoted to bounded variation ($BV$) functions in open sets of infinite dimensional separable
Banach spaces endowed with Gaussian measures.   $BV$ functions defined in the whole space $X$ 
have been introduced in \cite{F}  and studied also in \cite{FH,AMMP}.  As in the finite dimensional 
case, they are strongly related to geometric measure theory and in particular to the theory of 
perimeters, see \cite{AMP,H1,CasLunMirNov}. 

We use notation and results from \cite{boga}, to which we refer for the general theory of Gaussian 
measures. Our setting is an abstract Wiener space, i.e., a separable Banach space $X$, with a 
nondegenerate centred Gaussian measure $\gamma$ and the induced Cameron-Martin Hilbert space $H\subset X$
(with inner product $[\cdot, \cdot]_H$ and norm $|\cdot|_H$). See Subsection \ref{wiener} for more details. 

A basic result of the theory of $BV$ functions in Wiener spaces is the characterization of the $BV$ functions 
in $X$ as the elements of the Orlicz space $L \ln L^{1/2}(X,\gamma)$ such that 
\begin{equation}\label{BVonX}
\lim_{t\to 0}\int_X |\nabla_H T_t u(x)|_H d\gamma(x)<+\infty . 
\end{equation}
In this case, the gradient $D_\gamma u$ of $u$ is an $H$-valued measure on $X$ and the above limit is just 
the total variation of $D_\gamma u$ in $X$
(see for instance in \cite[Proposition 4.1]{F}, \cite[Proposition 3.6]{FH} and \cite[Theorem 4.1]{AMMP}). 
Here, $(T_t)_{t\geq 0}$ denotes the Ornstein--Uhlenbeck semigroup
\begin{equation}\label{OUsgrp}
T_t u(x)=\int_X u(e^{-t}x+\sqrt{1-e^{-2t}}y)d\gamma(y), \quad t>0, 
\end{equation}
that, as well known,  plays the role of the heat semigroup in the context of Wiener spaces. A main 
feature in the proof is the monotonicity of the function
$t\mapsto \int_X |\nabla_H T_tu(x)|_H d\gamma(x)$.

If $X=\R^n$ is endowed with the Lebesgue measure, under some regularity assumptions on  $\Omega$  there 
are bounded extension operators from $W^{1,p}(\Omega )$ and $BV(\Omega)$ to  $W^{1,p}(\R^n )$ and 
$BV(\R^n)$, respectively. In infinite dimensions, the lack of analogous results makes  the study of 
Sobolev and $BV$ functions on domains more difficult.

In this paper we deal with $BV$ functions (and sets with finite perimeter) on convex open sets of abstract 
Wiener spaces. The theory is still at its very beginning, see \cite{H1,H2,HinUch08,CelLunTra,bogaext}. 
However, convex open sets have finite perimeter by   \cite{CasLunMirNov}. 

We propose a definition of $BV$ functions on domains in $X$ through an integration by parts formula 
against a suitable class of test functions, rather than merely as restrictions of $BV$ functions on 
the whole space. This is not a trivial issue because of the lack of smooth bump functions (for general 
$X$) on the one hand, and of bounded extension operators on the other hand. 

Moreover, we relate the variation of a function and the short time behaviour of the Ornstein-Uhlenbeck 
semigroup; besides the interest of extending similar results available in different contexts, see 
\cite{AMPP,BMP,GP}, such a relation has proved to be useful to describe  fine properties of $BV$ functions 
in Wiener spaces, see \cite{AmbFig2,AmbFigRuna}.
 
In particular, we get a characterisation of $BV$ functions on a convex open set $\Omega$ similar to 
\eqref{BVonX}. We consider the Ornstein--Uhlenbeck operator $L$ associated to the Dirichlet form 
\[
\mathscr{E}(u,v)=\int_\Omega \scalH{\nabla_H u(x)}{\nabla_H v(x)} d\gamma(x), \quad u, v\in W^{1, 2}(\Omega, \mu), 
\]
and the semigroup $(T_t)_{t\geq 0}$ generated by the realization of $L$ in $L^2(\Omega, \gamma)$. 
The main result of this paper is the next Theorem. 

\begin{Theorem}\label{mainThm}
Let $\Omega\subset X$ be an open convex set, and let  $u_0\in BV(X,\gamma)\cap L^2(X,\gamma)$ be  
such that $|D_\gamma u_0|(\partial \Omega)=0$. Then, for any $t>0$,
\[
\int_\Omega |\nabla_H T_t u_0(x)|_H d\gamma(x) \leq |D_\gamma u_0|(\Omega)
\]
and
\begin{align}\label{main}
\lim_{t\to 0} \int_\Omega |\nabla_H T_tu_0(x)|_H d\gamma(x)= |D_\gamma u_0|(\Omega).  
\end{align}
\end{Theorem}

The study of the Ornstein--Uhlenbeck semigroup on domains is less straightforward than in the whole 
$X$, since no explicit formula like \eqref{OUsgrp} is available for $T(t)$; nevertheless, the function 
$t\mapsto \int_{\Omega} |\nabla_H T_tu(x)|_H d\gamma(x)$ is still monotone. In the proof of monotonicity  
the convexity of $\Omega$ plays an essential role.


The paper is organized as follows. In Section \ref{notations} we recall basic notations and results, both 
in the finite and infinite dimensional case. In Section \ref{3} we prove Theorem \ref{mainThm}. 
In the appendix we collect some useful properties of convex domains.


\section{Notations and preliminaries}          
 \label{notations}


In this section we introduce the basic notation and recall some preliminary results, both in finite and 
in infinite dimensional spaces. Let us recall that $u:\R^d\to\R$ is in 
$BV_{\rm loc}(\R^d)$ if $u\in L^1_{\rm loc}(\R^d)$ and 
\[
V(u,{\mathscr O})=\sup\Bigl\{\int_{\mathscr O} u(x)\sum_{j=1}^d \frac{\partial\varphi_j}{\partial x_j}(x)\,dx:
\ \varphi\in C_c^1({\mathscr O},\R^d),\ |\varphi(x)|\leq 1\ \forall\, x\in{\mathscr O}\Bigr\}<\infty 
\]
for every bounded open set ${\mathscr O}\subset\R^d$. In this case, the distributional gradient of 
$u$ is a locally finite Radon measure with values in $\R^d$. For more information about $BV$ functions 
and their fine properties in finite dimensions we refer to \cite{AmbFusPal00}. In the next subsection 
we describe our finite dimensional Gaussian framework and the related class of $BV$ functions. We refer 
to \cite{boga} for a comprehensive presentation of infinite dimensional gaussian analysis, to 
\cite{MirNovPal} for a survey on $BV$ functions in infinite dimensions. 

\subsection{$BV$ functions in finite dimensions} 

Let $\mathscr{O}$ be an open set in a $d$-dimensional Euclidean space $F$ endowed with the {\em standard} 
Gaussian measure, 
\[ 
d\gamma_F(y) = \frac{1}{(2\pi)^{d/2}}\exp(-|y|^2/2)dy = G_d(y)dy. 
\]
We define the  (negative)  formal adjoint of the partial derivative $\partial_{y_j}$ by 
$\partial_{y_j}^*\varphi(y)= \partial_{y_j}\varphi(y)-y_j\varphi(y)$,  and the Gaussian divergence by 
\begin{equation}   \label{divF}
{\rm div}_F\varphi=\sum_{j=1}^d\partial^*_{y_j}\varphi_j
,\qquad \varphi=\{\varphi_1,\ldots,\varphi_d\} \in C^1({\mathscr O},F).
\end{equation}

Given a function $u\in L^1(\mathscr{O},\gamma_F)$, we define its Gaussian total variation by
\begin{equation}\label{defvar}
|D_{\gamma_F} u|(\mathscr{O})=\sup 
\left\{\int_\mathscr{O} u(y){\rm div}_F \varphi (y) d\gamma_F(y):\ 
\varphi\in C^1_c(\mathscr{O},F), \;|\varphi(y)|\leq 1 \;\forall y\in \mathscr{O}\right\}.
\end{equation}
The space $BV(\mathscr{O},\gamma_F)$ is defined as the set of functions $u\in L^1(\mathscr{O},\gamma_F)$
with $|D_{\gamma_F}u|(\mathscr{O})<+\infty$. This space is characterised by the fact that the distributional
gradient of any $u\in BV(\mathscr{O},\gamma_F)$  is a vector valued finite measure $\mu=(\mu_1,\ldots,\mu_d)$, 
namely 
\[
\int_\mathscr{O} u(y) \partial^*_{y_i}g(y)d\gamma_F (y)
=\int_\mathscr{O} g(y)d\mu_i (y),\qquad \forall g\in C^1_c(\mathscr{O}), \;i=1,\dots,d.
\]
For such a measure we have   $|\mu|(\mathscr{O})=|D_{\gamma_F}u|(\mathscr{O})$; moreover
for any open set $A\subset \mathscr{O}$ we have
\[
|D_{\gamma_F} u|(A)=\sup \left\{\int_A u(y){\rm div}_F \varphi (y) d\gamma_F(y):\;
\varphi\in C^1_c(A,F),\; |\varphi(y)|\leq 1  \;\forall y\in A\right\}.
\]
The space $BV(\mathscr{O},\gamma_F)$ is a Banach space with the norm
\begin{equation}\label{norm}
\|u\|_{BV} =\|u\|_{L^1(\mathscr{O},\gamma_F)}+|D_{\gamma_F}u|(\mathscr{O}).
\end{equation}
We refer to \cite{CFMP} and to \cite{CMN} for more details on sets with finite perimeter and functions 
with bounded variation in Gaussian spaces. Let us point out that, due to the regularity and non degeneracy 
of the standard Gaussian density, local and fine properties of Gaussian $BV$ functions do not differ 
from those of classical $BV_{\rm loc}$ functions and indeed 
$BV({\mathscr O})\subset BV({\mathscr O},\gamma_F)\subset BV_{\rm loc}({\mathscr O})$. 
As a matter of fact, a Gaussian $BV$ function defines a finite measure that we denote by 
$D_{\gamma_F} u$ for which the following integration by parts formula holds, 
\[
\int_{\mathscr{O}} u(x) {\rm div}_F \varphi(x)d\gamma_F(x) =
\int_{\mathscr{O}} \langle\varphi, D_{\gamma_F} u\rangle,
\qquad \forall \varphi\in C^1_c(\mathscr{O};{\mathbb R}^d).
\]
Such a measure is absolutely continuous with respect to the standard total variation measure of $u$ and
$D_{\gamma_F} u=G_d Du$. Of course, global properties are very different, as the Gaussian density vanishes 
at infinity. Moreover, if $u\in W^{1,1}({\mathscr O},\gamma_F)$, then $u\in BV({\mathscr O},\gamma_F)$ and 
\begin{equation}\label{W11subsetBV}
|D_{\gamma_F}u|({\mathscr O})=\int_{\mathscr O} |\nabla u(y)| d\gamma_F(y). 
\end{equation} 
To some extent the above norm \eqref{norm} is too strong, for instance smooth functions are not dense in 
$BV(\mathscr{O},\gamma_F)$. However $BV$ functions may be approximated by smooth functions 
in the sense of the so called approximation in variation, as the next lemma shows. The proof we present 
is based on classical ideas that require a minor adaptation to the present situation.

Here and in the following, we denote by $B_R(x)$ the open ball centred at $x$ with radius $R$. 

\begin{Lemma}\label{AnzGiaW12}
Let $v_0\in BV(\mathscr{O},\gamma_F)\cap L^2({\mathscr O},\gamma_F)$. Then  for any $\varepsilon >0$ there exists
$v_\varepsilon\in W^{1,2}(\mathscr{O},\gamma_F)$ such that
\[
\| v_0-v_\varepsilon\|_{L^2(\mathscr{O},\gamma_F)}< \varepsilon
\qquad \text{and}\qquad 
\left|
\int_\mathscr{O} |\nabla v_\varepsilon(x)|d\gamma_F(x)-|D_{\gamma_F}v_0|(\mathscr{O})
\right|<\varepsilon .
\]
\end{Lemma}
\begin{proof}
The proof is a modification of the classical Meyers-Serrin arguments; we refer to 
\cite[Theorem 3.9]{AmbFusPal00} for more details. First of all we can assume that $v_0$ has compact 
support in $F$. Indeed, for every $R>1$ let  $\vartheta_R$ be a cut--off 
function such that $0\leq \vartheta_R \leq 1$,  $|\nabla \vartheta_R(x)|\leq 2$ for every $x$ and
\[
\vartheta_R\equiv 1 \mbox{ on }B_{R-1}(0),\qquad
{\rm supp}(\vartheta_R )\subset B_{R}(0). 
\]
Setting  $v_R=v_0 \vartheta_R$, we have
\[
\limsup_{R\to +\infty} \| v_0-v_R \|_{L^2(\mathscr{O},\gamma_F)}
\leq 
\lim_{R\to +\infty} \| v_0\|_{L^2(\mathscr{O}\setminus B_{R-1}(0),\gamma_F)} =0
\]
and, by the obvious equality 
${\rm div}_F(\vartheta_R \varphi) = \vartheta_R {\rm div}_F\varphi + \langle\nabla\vartheta_R,\varphi\rangle$, 
\[
|D_{\gamma_F} v_R| (\mathscr{O})
\leq
2\int_{\mathscr{O}\cap (B_{R}(0)\setminus B_{R-1}(0))} |v_0|d\gamma_F + 
|D_{\gamma_F}v_0| (\mathscr{O} \cap  B_{R}(0)).
\]
This implies that $v_R$ converges to $v_0$ in variation as $R\to\infty$. Precisely,  for any 
$\varepsilon>0$ there is $R>0$ such that $\| v_R-v_0\|_{L^2(\mathscr{O},\gamma_F)}<\varepsilon$
and
\[
\Big| |D_{\gamma_F}v_0|(\mathscr{O})-|D_{\gamma_F}v_R|(\mathscr{O})\Big|<\varepsilon.
\]
 From now on, we consider $v=v_R$ with this choice of $R$ in place of $v_0$. 
We can consider a sequence of open sets $\mathscr{O}_j$ with compact closure
in $\mathscr{O}$ such that any point of $\mathscr{O}$ belongs to at most four sets 
$\mathscr{O}_j$; a possible choice is  
\[
\mathscr{O}_1 =\{ x\in \mathscr{O}\cap B_2(0): {\rm dist}(x,\partial \mathscr{O})>1/2\}
\]
and for $j\geq 2$
\[
\mathscr{O}_j =\Bigl\{ x\in \mathscr{O}\cap B_{j+1}(0)\setminus \overline{B}_{j-1}(0): 
\frac{1}{j+1}<{\rm dist}(x,\partial \mathscr{O})<\frac{1}{j-1}\Bigr\}.
\]
Let $\{\varphi_j\}_j$ be a partition of unity associated with such a covering, and let $\varrho$ be 
a standard mollifier. For every $j\in \N$ fix $\varepsilon_j<\varepsilon$  such that 
${\rm supp} ((v\varphi_j)*\varrho_{\varepsilon_j})\subset \mathscr{O}_j$ and
\[
\|(v\varphi_j)*\varrho_{\varepsilon_j}- v\varphi_j\|_{L^2({\mathscr O},\gamma_F)}+ \int_\mathscr{O} 
|(v\nabla\varphi_j)*\varrho_{\varepsilon_j}-
v\nabla\varphi_j|
d\gamma_F <\frac{\varepsilon}{2^j} .
\]
The  approximation of $v$ is then defined by
\[
v_\varepsilon =\sum_{j\in \N} (v\varphi_j)*\varrho_{\varepsilon_j}. 
\]
$v_\varepsilon$  is indeed smooth, and
\[
\|v_\varepsilon-v\|_{L^2({\mathscr O},\gamma_F)} 
\leq 
\sum_{j\in \N} \|(v\varphi_j)*\varrho_{\varepsilon_j}-v\varphi_j\|_{L^2({\mathscr O},\gamma_F)}
<\varepsilon.
\]
On the other hand, see \cite[Proposition 3.2]{AmbFusPal00} 
\[
\nabla v_\varepsilon =\sum_{j\in \N} (\varphi_j Dv)*\varrho_{\varepsilon_j}+
\sum_{j\in \N} \left( (v\nabla\varphi_j)*\varrho_{\varepsilon_j}-v\nabla\varphi_j
\right).
\]
Here we have used the fact that $BV(\mathscr{O},\gamma_F)\subset BV_{\rm loc}(\mathscr{O})$, hence 
$Dv$ is a measure with locally finite variation and the convolution above is well defined. From these 
considerations, we then obtain
\begin{align*}
|D_{\gamma_F} v_\varepsilon|(\mathscr{O}) = &
\int_{\mathscr{O}} |\nabla v_\varepsilon(x)|d\gamma_F(x)
\\ 
\leq& \sum_{j\in \N} \int_{\mathscr{O}} 
|v \nabla\varphi_j*\varrho_{\varepsilon_j}(x)-v(x)\nabla\varphi_j(x)| G_d(x)dx
\\
&+\sum_{j\in \N} \int_{\mathscr{O}} dx \int_{B_\varepsilon(x)}
G_d(x)\varphi_j(y) \varrho_{\varepsilon_j}(x-y) d|Dv|(y) 
\\
\leq& \varepsilon + \sum_{j\in \N} \int_{\mathscr{O}} dx \int_{B_\varepsilon(x)}
\frac{G_d(x)}{G_d(y)}\varphi_j(y) \varrho_{\varepsilon_j}(x-y) G_d(y)d|Dv|(y) 
\\
\leq & \varepsilon +\sum_{j\in \N}  e^{\varepsilon_j R + \varepsilon_j^2/2} 
\int_{\mathscr{O}}\varphi_j(y)d|D_{\gamma_F}v|(y) 
\leq \varepsilon+e^{\varepsilon R + \varepsilon^2/2} |D_{\gamma_F}v|(\mathscr{O}),
\end{align*}
where we have used the fact that the support of $v$ is contained in  $B_R(0)$ and for 
$y\in B_{\varepsilon_j}(x)$ 
\[
\frac{G_d(x)}{G_d(y)}=e^{\frac{|y|^2-|x|^2}{2}}\leq e^{\varepsilon_j R + \varepsilon_j^2/2}.
\]
\end{proof}

\subsection{Abstract Wiener spaces} 
\label{wiener}

We consider  an infinite dimensional separable Banach space   $X$ (whose topological 
dual we denote by $X^*$), endowed with the Borel $\sigma$-algebra ${\mathscr B}(X)$ and a centred and 
non degenerate Gaussian measure $\gamma$ with nondegenerate covariance operator $Q\in \mathcal{L}(X^*,X)$ 
uniquely determined by the relation
\[
y^*(Qx^*) =\int_X x^*(x) y^*(x)d\gamma(x)
\qquad\forall x^*,\,y^*\in X^*.
\]
If we consider the operator $R:L^2(X,\gamma)\to X$ given by the Bochner integral
\[
R \varphi=\int_X x \varphi (x)\,d\gamma(x),\qquad \varphi \in L^2(X,\gamma), 
\]
it is easily seen that its adjoint $R^*:X^*\to L^2(X,\gamma)$ is just the embedding operator,  
$(R^*x^*)(x)= x^*(x)$, $x\in X$, and the equality $Q=RR^*$ follows. 

The Cameron-Martin space $H$ is given by $R({\mathcal H})$, where ${\mathcal H}$ is the closure of $X^*$
in $L^2(X,\gamma)$. It coincides with the set of all $h\in X$ such that there exists 
$\hat{h}\in {\mathcal H}$ for which
\[
\int_X x^*(x) \hat{h}(x) \,d\gamma(x) = x^*(h), \quad x^*\in X^*. 
\]
In this case, we have $R(\hat{h}) = h$,  and $R_{|\mathcal H} :{\mathcal H}\to H$ is an isometry if we endow $H$ 
with the norm  $|\cdot|_H$ associated with the inner product 
$[h, k]_H = \langle \hat{h}, \hat{k}\rangle_{L^2(X, \gamma)}$. 
The space  $Q(X^*)$ is dense in $H$, and  $H$ is continuously and densely embedded in $X$.  

The symbol ${\mathscr F}C_b^1(X)$ denotes the space of bounded continuously differentiable
cylindrical functions with bounded derivatives, that is, $u\in {\mathscr F}C_b^1(X)$ if
\[
u(x)=\varphi (x_1^*(x),\ldots,x_m^*(x))
\]
for some $\varphi\in C_b^1(\R^m)$ and $x_1^*,\ldots,x_m^*\in X^*$. 

We fix once and for all an orthonormal basis $(h_j)_{j\in\N}$ of $H$, with $h_j=Qh_j^,\ h_j \in X^*$ 
(such a basis exists, see \cite[Corollary 2.10.10]{boga}).  We denote by $\pi_m: X\to$ span$\{h_1, \ldots, h_m\}$ the projection 
$\pi_mx= \sum_{j=1}^m \hat{h}_j(x)h_j$ and by $X_m$, $X_m^\perp$ the range and the kernel of $\pi_m$, 
respectively. Note that the restriction of $\pi_m$ to $H$ is the orthogonal projection on the linear span of $h_1, \ldots, h_m$. 

$\pi_m$ induces the canonical factorisation $\gamma = \gamma_m \otimes \gamma_m^\perp$, 
where $\gamma_m = \gamma \circ \pi_m^{-1}$ and $\gamma^\perp_m = \gamma \circ (I-\pi_m)^{-1}$ are 
the pull--back measures on $X_m$ and  $X_m^\perp$. 

For every function $u\in L^1(X,\gamma)$ we define its canonical cylindrical approximations ${\mathbb{E}}_m u$ 
by 
\begin{equation}\label{conditionalexp}
{\mathbb E}_m u(x) =\int_{X}u(\pi_mx+(I-\pi_m)y)d\gamma(y)
=\int_{X^\perp_m}u(\pi_mx+y')d\gamma^\perp_m(y'),
\end{equation}
Then, $\lim_{m\to \infty}{\mathbb{E}}_m u =  u$ in $L^1(X,\gamma)$ and $\gamma$-a.e. (see e.g. \cite[Corollary 3.5.2]{boga}). 
Moreover ${\mathbb{E}}_m u$ is invariant under translations along all the vectors in $X^\perp_m$, 
hence ${\mathbb E}_mu(x)=v(\pi_mx)$ for some function $v$. 

\noindent
Let us recall the notation for the 
partial derivative along $h\in H$ and for  its formal adjoint. For $f\in C^1_b(X)$ we set 
\[
\partial_h f(x)=\lim_{t\to 0}\frac{f(x+th)-f(x)}{t},\qquad
\partial^*_hf(x)=\partial_hf(x)-f(x)\hat{h}(x)
\]
where $h=R\hat{h}\in H$ with $\hat{h}\in{\mathcal H}$.
The gradient along $H$, $\nabla_H f:X\to H$ of $f$ is defined as
\[
\nabla_Hf(x)=\sum_{j\in \N}\partial_{h_j} f(x)h_j
\]
and it is the unique element $y\in H$ such that, for every $h\in H$, $\partial_h f(x)=[y,h]_H$.
Notice that if $f(x)=g(\pi_mx)$ with $g\in C^1({\R}^m)$, then
\[
\partial _h f(x)=\langle\nabla g(\pi_mx),\pi_m h\rangle_{{\mathbb R}^m}.
\]
The operator $\partial_h^*$ defined by 
 $\partial_h^*\varphi = - \partial_h\varphi  + \hat{h}\varphi$
is (up to a change of sign) the formal adjoint
of $\partial_h$ with respect to $L^2(X,\gamma)$, namely
\[
\int_X \varphi\, \partial_h f\, d\gamma=-\int_X f\partial_h^*\varphi\,  d\gamma
\qquad\forall\varphi,\,f\in {\mathscr F} C^1_b(X).
\]
We define the space ${\mathscr F}C^1_b(X,H)$ of cylindrical $H$-valued
functions as the vector space spanned by functions $\varphi h$, where
$\varphi$ runs in ${\mathscr F} C^1_b(X)$ and $h$ in $H$. With this
notation, the divergence operator is defined for $\varphi\in {\mathscr F}C^1_b(X,H)$ as
\[
{\rm div}_\gamma \varphi =\sum_{j\in \N} \partial^*_j [\varphi ,h_j]_H,
\]
and we have the integration by parts formula 
\[
\int_X [\nabla_Hf,\varphi]_Hd\gamma = - \int_X f {\rm div}_\gamma\varphi\ d\gamma 
\qquad f\in {\mathscr F}C^1_b(X),\;\varphi\in {\mathscr F}C^1_b(X;H).
\]
If we fix a finite dimensional space $F\subset Q(X^*)\subset H$ with ${\rm dim}\ F=d$, we  identify $F$
with $\R^d$ and we denote by ${\rm div}_F$ the divergence on $F$ defined according to \eqref{divF} 
with respect to any orthonormal basis $\{h_1,\ldots,h_d\}$ of $F$. Moreover, since $F\subset H$, 
there is an orthogonal projection of $H$ onto $F$. According to Theorem 2.10.11 in 
\cite{boga} there is a unique (up to equivalence) measurable projection $\pi_F:X\to F$ which 
extends it. 

We denote by $\mathscr{M}(X,H)$ the space of all $H$-valued  measures $\mu$ with finite total 
variation on $\B(X)$. The total variation measure $|\mu|$ of  $\mu$ is defined by 
\[
|\mu|(B) = \sup \bigg\{ \sum_{j=1}^{\infty} |\mu(B_j)|_H:\; B = \bigcup_{j\in \N} B_h \bigg\}, 
\]
where  $B_j\in \B(X)$ for every $j$ and $B_j\cap B_i=\emptyset$ for $j\neq i$. Moreover, using the 
polar decomposition $\mu=\sigma|\mu|$, the total 
variation of $\mu$ can be obtained as
\begin{equation}\label{totalvar}
|\mu|(\Omega)=\sup\Bigl\{\int_\Omega [\sigma,\varphi]_Hd|\mu|:\ \varphi\in C_b(\Omega, H),\ 
\|\varphi\|_\infty\leq 1\Bigr\}.
\end{equation}
Indeed, in the real valued case this is a direct consequence of  the isometry between the space 
of real measures on and open set $\Omega$ and the dual space of $C_b(\Omega)$, see \cite[Section IV.6]{DS}. 
The finite dimensional case follows immediately because a vector-valued measure is just 
an $n$-tuple of real-valued measures. If $\mu$ has infinitely many components $\mu_j=[\mu,h_j]_H$, 
then $\mu\in\mathscr{M}(X,H)$ if and only if 
\[
\sup_m |(\mu_1,\ldots,\mu_m)|(X) < \infty.
\]
In fact, setting $\lambda=\sup_m |(\mu_1,\ldots,\mu_m)|$, the inequality $\lambda\leq|\mu|$ is obvious. 
Conversely, since $\mu_j\ll\lambda$ for every $j\in\N$, there is a sequence of $\lambda$-measurable 
functions $(f_j)$ such that $\mu_j=f_j\lambda$ and $\sum_{j=1}^m|f_j|^2\leq 1\ \lambda$-a.e. for every 
$m\in\N$, whence $\|(f_j(x))\|_{\ell^2}\leq 1\ \lambda$-a.e., 
$\sum_{j=1}^\infty f_jh_j\in L^1(X,\lambda;H)$ and 
$\mu=\sum_{j=1}^\infty f_jh_j\lambda \in {\mathscr{M}}(X,H)$. 

\begin{Definition}\label{defBVWiener}
Let $u\in L^2(X,\gamma)$. We say that $u$ has bounded variation in $X$ and we write $u\in BV(X,\gamma)$ 
if there exists $\mu \in \mathscr{M}(X,H)$ such that for any $\varphi\in {\mathscr F}C^1_b(X)$ we have
\begin{equation}\label{BVintbyparts}
\int_X u(x) \partial^*_j\varphi(x) d\gamma(x)=-\int_X \varphi(x) d\mu_j(x)\qquad \forall j\in \N,
\end{equation}
where $\mu_j=[h_j,\mu]_H$. In this case we set $D_{\gamma}u = \mu$. 
\end{Definition}

Even though in this paper we deal with   $BV$ functions defined 
in the whole space, it is interesting to point out that an intrinsic definition of $BV(\Omega,\gamma)$ 
is possible, using a suitable class of test functions. By \eqref{defvar}, we notice that in finite 
dimension the natural class of test functions is that of boundedly supported smooth functions. In 
infinite dimensions compactly supported smooth functions are not adequate and for $H$-valued measures 
the following result holds. 

\begin{Lemma}\label{lemmaSupLip0}
Let $\Omega\subset X$ be open and let $\mu\in {\mathscr M}(\Omega,H)$ be an $H$-valued Radon measure. 
Then, denoting by $|\mu|$ the total variation measure and using the polar decomposition $\mu=\sigma|\mu|$ 
we have  
\[
|\mu|(\Omega) = \sup\Bigl\{\int_\Omega [\sigma,\varphi]_Hd|\mu|:\ \varphi\in {\rm Lip}_0(\Omega, H),\ 
\|\varphi\|_\infty\leq 1\Bigr\},
\]
where ${\rm Lip}_0(\Omega,H)$ denotes the space of $H$-valued functions defined on $X$, Lipschitz 
continuous with respect to the $X$-norm and vanishing in $X\setminus\Omega$.
\end{Lemma}
\begin{proof}
We recall that in our framework all Borel measures on $X$ are Radon measures. Therefore, 
for every $\varepsilon>0$ there are a function $\varphi_\varepsilon\in C_b(\Omega,H)$ with 
$\|\varphi\|_\infty\leq 1$ and a compact set $K\subset\Omega$ with 
$|\mu|(\Omega\setminus K)<\varepsilon$ such that 
\[
|\mu|(\Omega) \leq \int_\Omega [\sigma,\varphi_\varepsilon]_H\, d|\mu| + \varepsilon 
\leq \int_K [\sigma,\varphi_\varepsilon]_H\, d|\mu| +2\varepsilon. 
\]
Let us now approximate $\sigma$: there is $\sigma_\varepsilon\in C(K,H)$ with finite dimensional 
range (just write $\sigma=\sum_j[\sigma,h_j]_Hh_j$ and take a suitable finite dimensional projection) 
such that  
\[
\|\sigma-\sigma_\varepsilon\|_{L^1(K,|\mu|)}<\varepsilon\qquad \text{and}\qquad
|\mu|(\Omega) \leq \int_K [\sigma_\varepsilon,\varphi_\varepsilon]_H\, d|\mu| + 3\varepsilon .
\]
Notice that, since $\sigma_\varepsilon$ has finite dimensional range, only finitely many 
components of $\varphi_\varepsilon=\sum_j[\varphi_\varepsilon,h_j]_Hh_j$ are involved in the above 
integral. We may therefore argue component by component to show that $\varphi_\varepsilon$ can be 
approximated by ${\rm Lip}_0$ functions uniformly on $K$. To this end, let us first remark that by the
Stone-Weierstrass theorem the class of the restrictions to $K$ of ${\mathscr F}C^1_b(X)$ functions 
is dense in $C_b(K)$, hence there is a function $g_\varepsilon\in {\mathscr F}C_b^1(X,H)$ with 
finite dimensional range such that $\|g_\varepsilon - \varphi_\varepsilon\|_{L^\infty(K)}<\varepsilon$.
Moreover, the function 
\[
f(x)= \Bigl(1-\frac{2}{\delta}{\rm dist}(x,K)\Bigr)^+,\qquad x\in H,
\]
with $\delta={\rm dist}(K,\partial\Omega)$, belongs to ${\rm Lip}_0(\Omega)$, so that, setting 
\[
G(h)=\left\{\begin{array}{ll}
h\qquad &\text{if }|h|_H\leq 1 
\\
\frac{h}{|h|_H} &\text{if }|h|_H >1
\end{array} 
\right.
,\qquad h\in H,
\]
we have that $\psi_\varepsilon(x)=f(x)(G\circ g_\varepsilon)(x)\in {\rm Lip}_0(\Omega,H)$, 
$\|\psi_\varepsilon\|_{L^\infty(K,|\mu|)}\leq 1$ and 
$\|\varphi_\varepsilon-\psi_\varepsilon\|_{L^\infty(K)}\leq 2\varepsilon$, whence 
\begin{align*}
|\mu|(\Omega) &\leq \int_K [\sigma_\varepsilon,\psi_\varepsilon]_H\, d|\mu| + 5\varepsilon 
\leq \int_\Omega [\sigma,\psi_\varepsilon]_H\, d|\mu| + 6\varepsilon
\\
&\leq \sup\Bigl\{\int_\Omega [\sigma,\psi]_Hd|\mu|:\ \psi\in {\rm Lip}_0(\Omega, H),\ 
\|\psi\|_\infty\leq 1\Bigr\}+6\varepsilon . 
\end{align*}
By the arbitrariness of $\varepsilon$, the proof is complete.
\end{proof}

In the next lemma we extend the integration by parts formula \eqref{BVintbyparts} with 
$u\in BV(X,\gamma)$ to ${\rm Lip}_0(\Omega,H)$ functions. 

\begin{Lemma}\label{intbypartsinLip}
For every $\varphi\in {\rm Lip}_0(\Omega,H)$ and $u\in BV(X,\gamma)$ the following equality holds:
\begin{equation}\label{intbypartsinLip1}
\int_{\Omega}u\, {\rm div}_\gamma\varphi\, d\gamma = - \int_{\Omega} [\varphi,D_\gamma u]_H .
\end{equation}
\end{Lemma}
\begin{proof} Let us show \eqref{intbypartsinLip1} arguing component by component. Fix $h_j$, an 
element of the given orthonormal basis in $H$,   and 
consider   the  projection $\pi_j:H\to {\rm span}\,h_j$, $\pi_j(x) = \hat{h}_j(x)h_j$. Then write $x=y+th_j$, $u_y(t)=u(y+th_j)$. Setting $X_j^\perp=(I-\pi_j)(X)$ 
and $\varphi_j=[\varphi,h_j]_H$,   we have
\begin{align*}
\int_X u\partial^*_j \varphi_j\, d\gamma &= 
\int_{X_j^\perp}d\gamma^\perp_{h_j^\perp}\int_{\R}u_y(t)\partial^*_t(\varphi_{j})_y(t)\, d\gamma_1(t)
\\
&=-\int_{X_j^\perp}d\gamma^\perp_{h_j^\perp}\int_{\R}(\varphi_j)_y(t)dD_{\gamma_1}u_y(t)
=-\int_X \varphi_j\, d[D_\gamma u,h_j] ,
\end{align*}
where we have used the notation $\gamma=\gamma_1\oplus\gamma^\perp_{h_j^\perp}$ for the 
factorization of $\gamma$ induced by the decomposition of $X$ into $\pi_j(X)\oplus X_j^\perp$ 
and in the second line for any $y\in X_j^\perp$ the integral on $\mathbb R$ is with respect to the 
measure $D_{\gamma_1}u_y$, the measure derivative of the section $u_y$ of $u$, see \cite{AMMP}. 
\end{proof}

An easy but useful consequence is the following lower semicontinuity property of the total 
variation, see also \cite[Proposition 2.5]{AmbFig2} for a different proof. 

\begin{Corollary}\label{lsc}
Let $u\in BV(X,\gamma)$ and let $\Omega\subset X$ be any open set such that 
$u_{|\Omega}\in L^2(\Omega,\gamma)$. If a sequence $(u_n)_n$ converges to $  u$ 
in $L^2(\Omega,\gamma)$, then 
\[
|D_\gamma u|(\Omega) \leq \liminf_{n\to \infty} |D_\gamma u_n|(\Omega).
\]
\end{Corollary}
\begin{proof} By Lemmas \ref{lemmaSupLip0} and \ref{intbypartsinLip}, 
\begin{align*}
|D_\gamma u|(\Omega) &= \sup\Bigl\{\int_\Omega [\varphi,D_\gamma u]_H:\ 
\varphi\in {\rm Lip}_0(\Omega, H),\ \|\varphi\|_\infty\leq 1\Bigr\}
\\
&= \sup\Bigl\{\int_{\Omega}u\, {\rm div}_\gamma\varphi\, d\gamma :\ 
\varphi\in {\rm Lip}_0(\Omega, H),\ \|\varphi\|_\infty\leq 1\Bigr\}.
\end{align*}
On the other hand, $\varphi\in {\rm Lip}_0(\Omega,H)$ the functional 
$u\mapsto \int_X u\, {\rm div}_\gamma\varphi\, d\gamma$ is continuous in $L^2(X,\gamma)$ and 
therefore, by Lemma \ref{intbypartsinLip}, the functional $u\mapsto |D_\gamma u|(\Omega)$ is 
lower semicontinuous in $L^2(X,\gamma)$, as it is the supremum of continuous functionals. 
\end{proof}

It is not hard to see that if $u\in BV(X,\gamma)$ then
\[
|D_{\gamma}u|(X) = \sup\Bigl\{ \int_X [\varphi,D_\gamma u]_H :\ 
\varphi\in {\mathscr F}C^1_b(X,H), \;|\varphi(x)|_H\leq 1 \;\forall x\in X\Bigr\}, 
\]
see \cite{AMMP}. A useful consequence of Lemma \ref{lemmaSupLip0} is that  
the canonical cylindrical approximations ${\mathbb{E}}_m u$ defined in \eqref{conditionalexp} 
(which are known to converge to $u$ in variation, i.e.
\[
|D_\gamma u|(X) =\lim_{m\to +\infty} |D_\gamma {\mathbb E}_m u|(X),
\]
see equality (34) in \cite{AMMP}), verify the  
inequality $|D_\gamma{\mathbb E}_mu|(A) \leq |D_\gamma u|(A)$ for all open sets $A$. In fact, 
\begin{align}\label{conditineq}
|D_\gamma{\mathbb E}_mu|(A)&=\sup\Bigl\{\int_A {\mathbb E}_mu\,  {\rm div}_\gamma\varphi\, d\gamma,\ 
\varphi\in {\rm Lip}_0(A,H),\ \|\varphi\|_\infty \leq 1\Bigr\}
\\ \nonumber
&=\sup\Bigl\{\int_A [\varphi,dD{\mathbb E}_mu]_H,\ 
\varphi\in {\rm Lip}_0(A,H),\ \|\varphi\|_\infty \leq 1\Bigr\}
\\ \nonumber
&=\sup\Bigl\{\int_A \langle\pi_m\varphi, dD{\mathbb E}_mu \rangle,\ 
\varphi\in {\rm Lip}_0(A,H),\ \|\varphi\|_\infty \leq 1\Bigr\}
\\ \nonumber 
&=\sup\Bigl\{\int_A \langle\pi_m\varphi, d \left(\pi_m D_\gamma u\right)\rangle,\ 
\varphi\in {\rm Lip}_0(A,H),\ \|\varphi\|_\infty \leq 1\Bigr\}
\\ \nonumber
&\leq\sup\Bigl\{\int_A [\varphi,dD_\gamma u]_H,\ 
\varphi\in {\rm Lip}_0(A,H),\ \|\varphi\|_\infty \leq 1\Bigr\}
\\ \nonumber 
&=\sup\Bigl\{\int_A u \, {\rm div}_\gamma\varphi\, d\gamma,\ 
\varphi\in {\rm Lip}_0(A,H),\ \|\varphi\|_\infty \leq 1\Bigr\}
=|D_\gamma u|(A).
\end{align}
 
\subsection{Sobolev spaces and the Ornstein--Uhlenbeck semigroup on convex domains}
\label{sobolev}

There are several equivalent ways of defining Sobolev spaces on Wiener spaces, see \cite[Section 5.2]{boga}. 
If $X$ is replaced by a domain $\Omega\subset X$, the equivalence of different definitions is not obvious. 
Here we adopt the definition of \cite{CelLunTra}, that works for sublevel sets $\Omega=\{ G<0\}$ of Sobolev 
functions $G\in W^{1,p}(X, \gamma)$ for some $p>1$. Since we are interested in a convex  $\Omega$, we fix 
any $x_0\in \Omega$ and we define the Minkowski function
\[
\m(x): =\inf\left\{ \lambda\geq 0 : x-x_0 \in \lambda (\Omega-x_0) \right\}
\]
which is Lipschitz continuous; then $\Omega =\{ G<0\}$ with $G(x) =\m(x)-1\in W^{1,p}(X, \gamma)$ for 
every $p>1$. By \cite[Lemma 2.2]{CelLunTra}, the operator  
$ {\rm Lip}(\Omega)\to L^2(\Omega, \gamma; H)$ defined by $ u\mapsto \nabla_H \tilde u_{|\Omega}$, 
where $\tilde u$ is any Lipschitz continuous extension of $u$ to the whole $X$, is closable. The space 
$W^{1,2}(\Omega,\gamma)$ is defined as the domain of its closure, still denoted by $\nabla_H$. Therefore, 
it is a Hilbert space for the inner product 
\[
\langle u, v\rangle _{W^{1,2}(\Omega,\gamma)} = 
\int_{\Omega} uv\,d\gamma  +  \int_{\Omega} [\nabla_Hu, \nabla_Hv]_H\, d\gamma 
\]
which induces the graph norm of $\nabla_H$. The associated quadratic form in the gradient, 
\[
\mathscr{E}(u,v)=\int_\Omega \scalH{\nabla_H u(x)}{\nabla_H v(x)} d\gamma(x), \qquad
u,v\in W^{1,2}(\Omega,\gamma), 
\]
is used to define the Ornstein--Uhlenbeck 
operator $L:D(L)\subset L^2(\Omega,\gamma)\to L^2(\Omega,\gamma)$ by setting
\begin{align*}
D(L)=\Bigl\{u\in W^{1,2}(\Omega,\gamma):&\ \exists f\in L^2(\Omega,\gamma) \mbox{ s.t. } 
\\
&\mathscr{E}(u,v)=-\int_\Omega fvd\gamma,\quad \forall v\in W^{1,2}(\Omega,\gamma) \Bigr\},
\end{align*}
and $Lu=f$. The operator $(L,D(L))$ is self-adjoint in $L^2(\Omega ,\gamma)$ and dissipative (namely, 
$\langle Lu, u\rangle_{L^2(\Omega ,\gamma)} \leq 0$ for every $u\in D(L)$), hence it is the 
infinitesimal generator of an analytic contraction semigroup $(T_t)_{t\geq 0}$ in $L^2(\Omega ,\gamma)$. 

For the moment we have considered only real valued functions. In the sequel we use also the 
complexification of $L$ in the space $L^2(\Omega,\gamma; \C)$, which is  the operator associated with  
the sesquilinear form $(u,v)\mapsto \int_\Omega [\nabla_H u, \overline{\nabla_H v}]_H d\gamma$ 
defined for $u$, $v\in W^{1,2}(\Omega,\gamma; \C)$.  The semigroup generated by the complexification 
agrees with $(T_t)_{t\geq 0}$ on real valued functions, and we use its representation formula as a Dunford 
integral along a complex path.


\section{Proof of Theorem \ref{mainThm}}
\label{3}


The proof of Theorem \ref{mainThm} is divided in several steps and each step is discussed in
a subsection.

\subsection{Monotonicity in finite dimensions} 
\label{finiteDim}

Let ${\mathscr O}$ be a convex open set with smooth boundary in a finite dimensional space $F$, with scalar product $\langle \cdot, \cdot\rangle$ and norm $|\cdot|$. 
We denote by $ \nu^{\mathscr O }(x)$ the exterior unit normal vector at $x\in \partial  {\mathscr O }$. 
Let $(T^F_t)_{t\geq 0}$ be the semigroup generated by the Ornstein-Uhlenbeck operator $L$ defined by 
the Dirichlet form 
\[
\mathscr{E}_{\mathscr O}(u,v)=\int_{\mathscr O} \langle\nabla u,\nabla v\rangle d\gamma_F, \qquad
u,v\in W^{1,2}({\mathscr O},\gamma_F), 
\]
as explained in Subsection \ref{sobolev}. By \cite{DPL},  $D(L)\subset W^{2,2}({\mathscr O}, \gamma_F)$ 
and the elements of $D(L)$ satisfy the Neumann boundary condition $\partial u/\partial \nu^{\mathscr O }=0$ 
at $\partial {\mathscr O}$.  Moreover for every $u\in D(L)$ we have 
$Lu(x) = \Delta u(x) - \langle x, \nabla u(x)\rangle$. Since $L$ is a realization of an elliptic operator 
with smooth coefficients, and the boundary of ${\mathscr O }$ is smooth, 
the function $(t,x)\mapsto T^F_tu_0(x)$ is smooth in $(0, +\infty)\times {\mathscr O}$ for every
$u_0\in L^2({\mathscr O},\gamma_F)$.  

For any $v_0\in L^2({\mathscr O},\gamma_F)$ let us   introduce the function 
${\mathcal F}_{v_0}:(0,+\infty)\to [0,+\infty]$ defined as
\[
{\mathcal F}_{v_0}(t)=\int_\mathscr{O} |\nabla T_t^F v_0(x)| d\gamma_F(x) . 
\]
Then the following result holds.

\begin{Proposition}
\label{mainpropFD}
For each $v_0\in BV(\mathscr{O},\gamma_F)\cap L^2(\mathscr{O},\gamma_F)$   the function 
${\mathcal F}_{v_0}$ is decreasing in $(0,\infty)$. Moreover 
\[
{\mathcal F}_{v_0}(t)\leq |D_{\gamma_F}v_0|(\mathscr{O}),\qquad \forall t>0
\]
and
\[
\lim_{t\to 0} {\mathcal F}_{v_0}(t)= |D_{\gamma_F}v_0|(\mathscr{O}).
\]
\end{Proposition}
\begin{proof}
In order to avoid integrability problems, we introduce a family of cut--off functions $\vartheta_R$ such that
$0\leq \vartheta_R\leq 1$, $\vartheta_R\equiv 1$ in $B_R(0)$, ${\rm supp}(\vartheta_R)\subset B_{2R}(0)$ and 
$|\nabla \vartheta_R(x)|\leq 2/R$ for every $x$. Analogously, in order to overcome the lack of regularity 
of the function $|\nabla T^F_t v_0(x)|$ at its zeroes, we replace it by $\sqrt{|\nabla T_t^F v_0(x)|^2+1/R}$. 
We then define
\[
{\mathcal F}_{R,v_0}(t)=\int_\mathscr{O} \vartheta_R(x)
\sqrt{|\nabla T_t^F v_0(x)|^2+\frac{1}{R}} \, d\gamma_F(x).
\] 
As a first step, we prove that  ${\mathcal F}_{R,v_0}$ is differentiable. Since $T^F_t$ is analytic, 
$t\mapsto T^F_t v_0$ is differentiable with values in $D(L)$, and 
\[
\partial_t \sqrt{ |\nabla T^F_t v_0(x)|^2+\frac{1}{R}} =
\frac{1}{\sqrt{|\nabla T_t^F v_0(x)|^2+\frac{1}{R}}} \langle\nabla T^F_t v_0(x), \nabla L T^F_tv_0(x)\rangle
\]
so that 
\[
\partial_t \sqrt{  |\nabla T^F_t v_0(x)|^2+\frac{1}{R}} \leq
|\nabla L T^F_tv_0(x)|.
\]
Then we can differentiate under the integral, and recalling that 
$\partial_i\partial^*_j\varphi=\partial^*_j\partial_i\varphi-\varphi\delta_{ij}$ we get  
\allowdisplaybreaks
\begin{align*}
{\mathcal F}'_{R,v_0}(t)=&
\int_\mathscr{O} \frac{\vartheta_R(x)}{\sqrt{|\nabla T_t^F v_0(x)|^2+\frac{1}{R}}}
\langle\nabla T^F_t v_0(x), \nabla L T^F_t v_0(x)\rangle
d\gamma_F(x) \\
=&
\sum_{i,j=1}^d 
\int_\mathscr{O} \frac{\vartheta_R(x)}{\sqrt{|\nabla T_t^F v_0(x)|^2+\frac{1}{R}}}
\partial_i T^F_t v_0(x) \partial_i \partial^*_j \partial_j  T^F_tv_0(x)
d\gamma_F(x) \\
=&
\sum_{i,j=1}^d 
\int_\mathscr{O} \frac{\vartheta_R(x)}{\sqrt{|\nabla T_t^F v_0(x)|^2+\frac{1}{R}}}
\partial_i T^F_t v_0(x) \partial^*_j\partial^2_{ij}  T^F_t v_0(x)
d\gamma_F(x) +\\
&-\int_\mathscr{O} \vartheta_R(x)\frac{|\nabla T^F_t v_0(x)|^2}{\sqrt{|\nabla T_t^F v_0(x)|^2+\frac{1}{R}}}d\gamma_F(x)
\\
=&
\sum_{i,j=1}^d 
\int_{\partial \mathscr{O}} \frac{\vartheta_R(x)}{\sqrt{|\nabla T_t^F v_0(x)|^2+\frac{1}{R}}}
\partial_i T^F_t v_0(x) \partial^2_{ij}  T^F_t v_0 (x) \nu^\mathscr{O}_j(x) 
G_d(x)d\mathscr{H}^{d-1}(x)\\
&-
\sum_{i,j=1}^d 
\int_\mathscr{O} \partial^2_{ij} T^F_t v_0(x) 
\partial_j \left(
\frac{\vartheta_R(x)}{\sqrt{|\nabla T_t^F v_0(x)|^2+\frac{1}{R}}}
\partial_i T^F_t v_0(x)\right) d\gamma_F(x) \\
& -\int_\mathscr{O} \vartheta_R(x)\frac{|\nabla T^F_t v_0(x)|^2}{\sqrt{|\nabla T_t^F v_0(x)|^2+\frac{1}{R}}}
d\gamma_F(x)\\
=&
\int_{\partial \mathscr{O}} \frac{\vartheta_R(x)}{\sqrt{|\nabla T_t^F v_0(x)|^2+\frac{1}{R}}}
\langle D^2 T^F_tv_0(x) \nu^\mathscr{O}(x), \nabla T^F_tv_0(x)\rangle  
G_d(x)d\mathscr{H}^{d-1}(x)
\\
&
-\sum_{i,j=1}^d 
\int_\mathscr{O}
\frac{\vartheta_R(x)}{\sqrt{|\nabla T_t^F v_0(x)|^2+\frac{1}{R}}} \Big[
(\partial^2_{ij} T^F_t v_0(x))^2 +
\\
&-
\frac{1}{|\nabla T_t^F v_0(x)|^2+\frac{1}{R}} \partial^2_{kj} T^F_t v_0(x) 
\partial_k T^F_t v_0(x)\partial_i T^F_t v_0 (x)\partial^2_{ij} T^F_t v_0(x) 
\Big]d\gamma_F(x)+
\\
&
-\int_\mathscr{O} \frac{1}{\sqrt{|\nabla T_t^F v_0(x)|^2+\frac{1}{R}}}
\langle D^2 T_t^F v_0(x) \nabla \vartheta_R(x),\nabla T^F_t v_0(x)\rangle  d\gamma_F(x) +\\
&-\int_\mathscr{O} \vartheta_R(x) \frac{|\nabla T^F_t v_0(x)|^2}{\sqrt{|\nabla T_t^F v_0(x)|^2+\frac{1}{R}}}
d\gamma_F(x)
\\
=&
-\int_{\partial \mathscr{O}} \frac{\vartheta_R(x)}{\sqrt{|\nabla T_t^F v_0(x)|^2+\frac{1}{R}}} 
\langle J \nu^\mathscr{O}(x) \nabla T^F_t v_0 (x), \nabla T^F_t v_0(x)\rangle 
G_d(x)\, d\mathscr{H}^{d-1}(x)
\\
&+\int_\mathscr{O} \frac{\vartheta_R(x)}{\sqrt{|\nabla T_t^F v_0(x)|^2+\frac{1}{R}}}
\Big( 
\frac{|\nabla T^F_t v_0(x)|^2}{|\nabla T_t^F v_0(x)|^2+\frac{1}{R}}
\left|
D^2 T^F_t v_0(x) \frac{\nabla T^F_t v_0(x)}{|\nabla T^F_t v_0(x)|}\right|^2 +\\
&- \|D^2 T^F_t v_0(x)\|^2_2 \Big) d\gamma_F(x) +
\\
&-\int_\mathscr{O} \frac{1}{\sqrt{|\nabla T_t^F v_0(x)|^2+\frac{1}{R}}} 
\langle D^2 T_t^F v_0(x) \nabla \vartheta_R(x),
\nabla T^F_t v_0(x) \rangle d\gamma_F(x)  +\\
&-\int_\mathscr{O} \vartheta_R(x) \frac{|\nabla T^F_t v_0(x)|^2}{\sqrt{|\nabla T_t^F v_0(x)|^2+\frac{1}{R}}}
d\gamma_F(x),
\end{align*}
where we have denoted by $\|D^2 T^F_t v_0(x)\|$ the Euclidean norm of the matrix $D^2 T^F_t v_0(x)$.\\
The second integral in the right hand side is negative because 
\[
\left|
D^2 T^F_tv_0(x) \frac{\nabla T^F_tv_0(x)}{|\nabla T^F_t v_0(x)|}
\right|^2
\]
is bounded   by the square of the largest eigenvalue of $D^2 T^F_t v_0(x)$, while
$\|D^2 T^F_t v_0(x)\|^2_2$ is the sum of the square of all the eigenvalues. 
In the first integral we have denoted by $J \nu^\mathscr{O}(x)$ the 
Jacobian matrix of $  \nu^\mathscr{O}$ at $x$, and we have used the fact that $\nabla T^F_t v_0(x)$ is 
orthogonal to $\partial \mathscr{O}$ and any  tangential derivative of 
$\langle \nabla T^F_t v_0 (x), \nu^\mathscr{O}(x)\rangle$ is equal to $0$, that is
\begin{align*}
0=&
\langle \nabla \left(\langle \nabla T^F_t v_0 (x), \nu^\mathscr{O} (x)\rangle \right),  
\nabla T^F_t v_0(x)\rangle 
\\
=& \langle D^2 T^F_t v_0(x) \nu^\mathscr{O}(x), \nabla T^F_t v_0(x)\rangle  
+ \langle J \nu^\mathscr{O}(x)\nabla T^F_t v_0(x),  \nabla T^F_tv_0(x)\rangle.
\end{align*}
The convexity  of $\partial\mathscr{O}$ implies 
\[
\langle J \nu^\mathscr{O} \xi ,  \xi\rangle  \geq 0,\qquad \forall \xi \in (\nu^\mathscr{O}(x))^\perp
\]
and so we can conclude that 
\begin{align*}
{\mathcal F}'_{R,v_0}(t)\leq&
-\int_\mathscr{O}  \frac{1}{\sqrt{|\nabla T_t^F v_0(x)|^2+\frac{1}{R}}}
\langle D^2 T_t^F v_0(x) \nabla \vartheta_R(x), 
\nabla T^F_t v_0(x)\rangle  d\gamma_F(x)  \\
\leq &
\frac{2}{R} \|\, |D^2 T^F_t v_0|\, \|^2_{L^2(\mathscr{O},\gamma_F)},
\end{align*}
where the last inequality holds with $|D^2 T^F_t v_0|$ the operator norm of $D^2 T^F_t v_0$. \\ 
As a consequence, for any $t_1<t_2$
\[
{\mathcal F}_{R,v_0}(t_2) = {\mathcal F}_{R,v_0}(t_1)+\int_{t_1}^{t_2} {\mathcal F}'_{R,v_0}(s)ds . 
\leq {\mathcal F}_{R,v_0}(t_1)+
\frac{2}{R} \int_{t_1}^{t_2} \| \,|D^2 T^F_s v_0|\, \|^2_{L^2(\mathscr{O},\gamma_F)}ds;
\]
Letting $R\to +\infty$ we   obtain the monotonicity of ${\mathcal F}_{v_0}$, since
\[
{\mathcal F}_{v_0}(t_2)=\lim_{R\to +\infty}{\mathcal F}_{R,v_0}(t_2)
\leq 
\lim_{R\to +\infty} \left( {\mathcal F}_{R,v_0}(t_1)+
\frac{2}{R}\int_{t_1}^{t_2} \|\, |D^2 T^F_s v_0|\, \|^2_{L^2(\mathscr{O},\gamma_F)}ds \right)
={\mathcal F}_{v_0}(t_1).
\]
To prove the second part of the statement, let us fix $w\in W^{1,2}({\mathscr O},\gamma_F)$. 
Since $ W^{1,2}(\mathscr{O},\gamma_F) $ is the domain of $(I-L)^{1/2}$, then $T^F_t $ is strongly 
continuous in $ W^{1,2}(\mathscr{O},\gamma_F) $. It follows that $|\nabla T^F_t w|$ converges 
to $|\nabla  w|$ in $L^1(\mathscr{O}, \gamma)$  as $t\to 0$, and hence by \eqref{W11subsetBV}
\[
\lim_{t\to 0} \int_\mathscr{O} |\nabla T^F_t w(x)|d\gamma_F(x)=  
\int_\mathscr{O} |\nabla   w(x)|d\gamma_F(x) = |D_{\gamma_F} w|(\mathscr{O}).
\]
Therefore, for any $t>0$ and for $w\in W^{1,2}(\mathscr{O},\gamma_F)$
\[
{\mathcal F}_w (t)\leq {\mathcal F}_w(0^+)=\lim_{t\to 0} \int_\mathscr{O} |\nabla T^F_tw(x)|d\gamma_F(x)
=|D_Fw|(\mathscr{O}). 
\]
Now let $v_0\in BV(\mathscr{O},\gamma_F) \cap L^2(\mathscr{O},\gamma_F)$. 
Thanks to Lemma~\ref{AnzGiaW12}, there exists a sequence of functions 
$(w_j)_{j\in \N}\subset W^{1,2}(\mathscr{O},\gamma_F)$ such that 
\[
\lim_{j\to \infty}\|w_j-w_0\|_{L^2(\mathscr{O},\gamma_F)} =0, \quad 
\lim_{j\to \infty}\int_\mathscr{O}  |\nabla w_j (x)|d\gamma_F(x) = |D_{\gamma_F}v_0|(\mathscr{O}). 
\]
Then, for every $t>0$ we have 
$\lim_{j\to \infty} |\nabla T^F_t w_j  | =  |\nabla T^F_t v_0|$ in $L^2(\mathscr{O},\gamma_F)$, 
and by the first part of the proof 
\[
\int_\mathscr{O} |\nabla T^F_t w_j (x)|d\gamma_F(x) \leq  
\int_\mathscr{O} |\nabla   w_j (x)|d\gamma_F(x) , \quad j\in \N, 
\]
so that 
\begin{align*}
\int_\mathscr{O} |\nabla T^F_t v_0 (x)|d\gamma_F (x)
& = \lim_{j\to +\infty} \int_\mathscr{O} |\nabla T^F_t w_j (x)|d\gamma_F(x) \\
&\leq
\liminf_{j\to +\infty} \int_\mathscr{O} |\nabla w_j (x)|d\gamma_F(x) 
= |D_{\gamma_F} v_0|(\mathscr{O}).
\end{align*}
\end{proof}

\begin{Remark}
Some arguments of this section may be related to  the approach of Bakry and \'Emery \cite{bakem}, 
that has been widely developed in the last years also in the metric space setting, and in fact 
  results of this type in such framework can be found in \cite{Sav}, to which we refer for the details. 
\end{Remark}

\subsection{Convex sets as countable intersection of cylindrical convex sets}
\label{approxcx}

In this section we consider a convex open set $\Omega\subset X$ and a sequence $(\Omega_n)_n$ 
of open convex cylindrical sets of the form $\Omega_n=\pi_{F_n}^{-1}(\mathscr{O}_n)$, where, 
for every $n\in{\mathbb N}$, $\pi_{F_n}$ is a finite dimensional projection from $X$ onto $F_n$, such that  
$\Omega_n\subset \Omega_{n+1}$, $\partial \mathscr{O}_n$ is regular, $\Omega\subset \Omega_n$ and 
\[
\overline \Omega=\bigcap_{n\in \N} \overline \Omega_n.
\]
We give a construction of the approximating sets $\Omega_n$ in the Appendix.\\
Since $\Omega$ and  $\Omega_n$ are  open convex sets, by \cite[Prop. 4.2]{CasLunMirNov} we know that 
$\gamma(\partial \Omega)=\gamma(\partial \Omega_n)=0$. Moreover, 
$\Omega\subset \Omega_n$ for any $n\in \N$ and 
\[
\gamma\left( 
\bigcap_{n\in \N}\Omega_n \setminus \Omega
\right)=0.
\]

\subsection{Approximations of semigroups}

We recall that $L$ is the Ornstein-Uhlenbeck operator in $L^2(\Omega, \gamma)$ defined in 
Section \ref{sobolev}. Similarly, $L_n$ are the Ornstein-Uhlenbeck operators in $L^2(\Omega_n, \gamma)$ 
associated with the Dirichlet forms 
\[
\mathscr{E}^{(n)}(u,v)=\int_{\Omega_n} \scalH{\nabla_H u}{\nabla_H v}d\gamma, \quad 
u, \;v\in W^{1,2}(\Omega_n, \gamma).
\]
The semigroups generated by $L_n$ are denoted by $(T^{(n)}_t)_{t\geq 0}$. We denote by $R(\lambda,A)$ the 
resolvent of the operator $A$. In the next proposition we deal with complex-valued 
functions.

\begin{Proposition}\label{approxSG}
Under the above assumptions, for any $f\in L^2(X,\gamma)$ and for any 
$\lambda\in \C\setminus (-\infty,0]$, 
\[
\lim_{n\to \infty} \Big(R(\lambda,L_n) (f_{|\Omega_n})\Big)_{|\Omega} = 
R(\lambda,L) (f_{|\Omega})\qquad \mbox{in } W^{1,2}(\Omega,\gamma).
\]
It follows  
\[
(T^{(n)}_tu_{0|\Omega_n})_{|\Omega}  \to T_t u_{0|\Omega}  \qquad \mbox{in } W^{1,2}(\Omega,\gamma)
\]
for any $u_0\in L^2(X,\gamma)$ and $t>0$.
\end{Proposition}
\begin{proof}
Fix $\lambda \in \C\setminus (-\infty,0]$ and set $\vartheta ={\rm arg} \lambda$.  
Since each $L_n$ is self--adjoint and dissipative, $\lambda\in \rho(L_n)$ and  
\[
\|R(\lambda,L_n)\|_{\mathcal{L}(L^2(\Omega_n,\gamma))}\leq \frac{1}{|\lambda|\cos\frac{\vartheta}{2}}, 
\]
\[
\|L_n R(\lambda,L_n)\|_{\mathcal{L}(L^2(\Omega_n,\gamma))}
=\| \lambda R(\lambda,L_n)-I\|_{\mathcal{L}(L^2(\Omega_n,\gamma))}
\leq \frac{1}{\cos\frac{\vartheta}{2}}+1. 
\]
Setting $u=R(\lambda,L) (f_{|\Omega})$ and $u_n=R(\lambda,L_n)(f_{|\Omega_n})$, such estimates imply
\[
\|u_n\|_{L^2(\Omega_n,\gamma)} 
\leq
\frac{1}{|\lambda| \cos\frac{\vartheta}{2}}\|f\|_{L^2(X,\gamma)} 
\]
and 
\[
\int_{\Omega_n} |\nabla_H u_n|^2_H d\gamma 
= -\int_{\Omega_n} u_n L_n u_n d\gamma  
\leq 
\frac{1+\cos\frac{\vartheta}{2}}{|\lambda| \cos^2\frac{\vartheta}{2}} \|f\|^2_{L^2(X,\gamma)}. 
\]
Recalling  that $\Omega\subset \Omega_n$, we arrive at the estimate
\begin{align*}
\|u_{n|\Omega}\|_{W^{1,2}(\Omega,\gamma)} 
\leq
\|u_n\|_{W^{1,2}(\Omega_n,\gamma)} 
\leq
\frac{\|f\|_{L^2(X,\gamma)}}{\cos\frac{\vartheta}{2}} 
\left(
\frac{1}{|\lambda|}+\sqrt{\frac{1+\cos\frac{\vartheta}{2}}{|\lambda|}}
\right).
\end{align*}
Then the sequence $(u_{n|\Omega})$ is bounded in $W^{1,2}(\Omega,\gamma)$, so that  it admits a 
subsequence weakly convergent to some function $u_\infty$.
By definition,  $u_n=R(\lambda,L_n)(f_{|\Omega_n})$ means
\begin{equation}
\label{defSolN}
\lambda\int_{\Omega_n} u_n \varphi \,d\gamma + \int_{\Omega_n} \scalH{\nabla_H u_n}{\nabla_H \varphi}d\gamma
=\int_{\Omega_n} f\varphi \, d\gamma,\qquad \forall \varphi \in W^{1,2}(\Omega_n,\gamma).
\end{equation}
Fix any $\varphi \in W^{1,2}(X,\gamma)$. Notice that
\[
\left|
\int_{\Omega_n\setminus \Omega} f\varphi \, d\gamma\right|
\leq \| f\|_{L^2(X,\gamma)} \| \varphi\|_{L^2(\Omega_n\setminus\Omega,\gamma)}
\]
where $\lim_{n\to \infty} \| \varphi\|_{L^2(\Omega_n\setminus\Omega,\gamma)} =0$ since 
$\gamma(\Omega_n\setminus\Omega)$ vanishes as $n\to \infty$. Then, 
\[
\lim_{n\to+\infty}
\int_{\Omega_n\setminus \Omega} f\varphi \, d\gamma=0.
\]
Recalling that  $\|u_n\|_{W^{1,2}(\Omega_n,\gamma)}$ is bounded by a constant independent of $n$, 
the same  argument yields 
\[
\lim_{n\to+\infty}
\int_{\Omega_n\setminus \Omega} u_n\varphi \,d\gamma=
\lim_{n\to+\infty}
\int_{\Omega_n\setminus \Omega} \scalH{\nabla_H u_n}{\nabla_H\varphi} d\gamma= 0.
\]
We   conclude that
\begin{align*}
\lambda \int_\Omega u_\infty \varphi \, d\gamma 
+\int_{\Omega} \scalH{\nabla_H u_\infty}{\nabla_H \varphi} d\gamma 
&= \lim_{n\to+\infty} \left(
\lambda \int_{\Omega_n} u_n \varphi \,d\gamma +\int_{\Omega_n} \scalH{\nabla_H u_n}{\nabla_H \varphi} d\gamma 
\right) \\
&=
\lim_{n\to+\infty}  \int_{\Omega_n} f\varphi \, d\gamma =\int_\Omega f\varphi \, d\gamma.
\end{align*}
Since the restrictions to $\Omega$ of elements of $W^{1,2}(X,\gamma)$ are dense in $W^{1,2}(\Omega,\gamma)$, 
we obtain
\[
\lambda \int_\Omega u_\infty \varphi \, d\gamma 
+\int_{\Omega} \scalH{\nabla_H u_\infty}{\nabla_H \varphi} d\gamma 
=\int_\Omega f\varphi \, d\gamma, \quad \forall \varphi \in W^{1,2}(\Omega,\gamma). 
\]
Therefore, the limit function $u_\infty$ coincides   with $u= R(\lambda, L)(f_{|\Omega})$ and  the whole 
sequence $u_{n|\Omega}$ weakly converges in $W^{1,2}(\Omega,\gamma)$ to $u$, with no need of subsequences.  

Let us now show that $u_{n|\Omega}$ converges strongly to $u$ in $W^{1,2}(\Omega,\gamma)$. To this aim
it is enough to show that 
\[
\limsup_{n\to+\infty} \|u_n\|_{W^{1,2}(\Omega,\gamma)}\leq \| u\|_{W^{1,2}(\Omega,\gamma)}.
\]
To see this we use \eqref{defSolN} with $\varphi=\bar u_n$ and obtain 
\[
\lambda \int_{\Omega_n} |u_n|^2 d\gamma +\int_{\Omega_n} |\nabla_H u_n|^2_H d\gamma
=\int_{\Omega_n} f\bar u_n d\gamma.
\]
Letting $n\to +\infty$, as before we obtain
\begin{equation}    \label{serve}
\begin{array}{lll}
\displaystyle{\lim_{n\to+\infty} \left(
\lambda \int_{\Omega_n} |u_n|^2 d\gamma +\int_{\Omega_n} |\nabla_H u_n|^2_H d\gamma 
\right)}
& =& \displaystyle{\int_{\Omega} f\bar u d\gamma }\\
\\
& =& 
\lambda \displaystyle{\int_\Omega |u|^2d\gamma +\int_\Omega |\nabla_H u|^2_H
d\gamma }
\end{array}
\end{equation}
since  $\lambda u-Lu=f$.

We write $\lambda=\alpha+i\beta$ with $\alpha$, $\beta\in \R$. 
If $\beta \neq 0$, taking the imaginary parts in \eqref{serve} we get
\[
\lim_{n\to+\infty} \beta \int_{\Omega_n} |u_n|^2d\gamma =\beta \int_\Omega |u|^2d\gamma,
\]
that is
\[
\lim_{n\to+\infty} \int_{\Omega_n} |u_n|^2d\gamma =\int_\Omega |u|^2d\gamma ,
\]
and we deduce 
\[
\lim_{n\to+\infty} \int_{\Omega_n} |\nabla_Hu_n|_H^2d\gamma =\int_\Omega |\nabla_H u|_H^2d\gamma.
\]
Therefore, 
\begin{align*}
\limsup_{n\to+\infty} \left(
\int_{\Omega} |u_n|^2 d\gamma +\int_{\Omega} |\nabla_H u_n|^2_H d\gamma
\right)
&\leq
\limsup_{n\to+\infty} \left(
\int_{\Omega_n} |u_n|^2 d\gamma +\int_{\Omega_n} |\nabla_H u_n|^2_H d\gamma
\right) \\
&=
\int_\Omega |u|^2d\gamma +\int_\Omega |\nabla_H u|^2_Hd\gamma .
\end{align*}
If $\beta=0$, since $\lambda\in {\mathbb C}\setminus (-\infty,0]$, we have $\alpha>0$;
\eqref{serve} gives 
\[
\lim_{n\to \infty} \alpha \|u_n\|_{L^2(\Omega_n,\gamma)}^2+\|\nabla_H u_n\|_{L^2(\Omega_n,\gamma)}^2
=\alpha \|u\|_{L^2(\Omega,\gamma)}^2+\|\nabla_H u\|^2_{L^2(\Omega,\gamma)}, 
\]
and since the norm $u\mapsto (\alpha \|u\|_{L^2(\Omega,\gamma)}^2+\|\nabla_H u\|_{L^2(\Omega,\gamma)}^2)^{1/2}$ 
is equivalent to the norm of $W^{1,2}(\Omega, \gamma)$ we are done. 
 
Convergence of resolvents implies convergence of semigroups. Indeed, it is sufficient to use the Dominated 
Convergence Theorem in  the canonical representation  formula, 
\[
T_t f_{|\Omega} = \frac{1}{2\pi i}\int_{\Gamma}R(\lambda,L) f_{|\Omega}\,d\lambda
=\lim_{n\to +\infty}\frac{1}{2\pi i}\int_{\Gamma} R(\lambda,L_n) f_{|\Omega_n}\, d\lambda
=\lim_{n\to +\infty} T^{(n)}_t f_{|\Omega_n}, 
\]
where $\Gamma$ is any of the usual integration paths for analytic semigroups. 
\end{proof}

\subsection{Conclusion: approximation by finite dimensional estimates}       \label{Conclusion}

In this Subsection we complete the proof of Theorem \ref{mainThm}. 

\begin{proof}
First of all, by lower semicontinuity (Corollary \ref{lsc}) we know that 
\[
|D_\gamma u_0|(\Omega)\leq \liminf_{t\to 0}|D_\gamma T_tu_0|(\Omega) = 
\liminf_{t\to 0}\int_\Omega |\nabla_H T_tu_0|_H d\gamma
\]
by the strong continuity in $L^2(\Omega,\gamma)$ of the semigroup $(T_t)_{t\geq 0}$.

Next, we prove  the estimate
\[
\int_{\Omega_n} |\nabla_H T^{(n)}_t u_{0|\Omega_n}|_H d\gamma \leq |D_\gamma u_0|(\Omega_n)
\]
where $\Omega_n$ is the approximation of $\Omega$ constructed in the Appendix, $(T^{(n)})_{t\geq 0}$ 
is the semigroup associated with the Dirichlet form ${\mathcal E}^{(n)}$ in $L^2(\Omega_n, \gamma)$. 
 
Let $v_j={\mathbb E}_ju$ be the sequence of canonical cylindrical approximations of $u_0$, converging 
to $u_0$ in variation. Fixed any $n$, $j$ we choose a finite 
dimensional space $F\subset Q(X^*)$  such that  $\mathscr{O}_n \subset F$ and
$v_j(x)=w_j(\pi_F(x))$ with $w_j:F\to \R$. Then, we have  the equality
\[
T^{(n)}_t v_{j|\Omega_n}=T^F_t (w_j\circ \pi_F)_{|\mathscr{O}_n}
\]
where $(T^F_t)_{t\geq 0}$ is the semigroup associated with the Dirichlet form
\[
\int_{\mathscr{O}_n} \langle\nabla u, \nabla v\rangle d\gamma_F 
\]
in $L^2({\mathscr{O}}_n, \gamma_F)$. This follows from the fact that the function
\[
g:[0,  \infty) \to L^2(\Omega_n, \gamma), \quad g(t)=T^{F}_t (w_j\circ\pi_F)_{|\mathscr{O}_n}
\]
belongs to 
$C([0, \infty);L^2(\Omega_n,\gamma))\cap C^1((0,\infty);L^2(\Omega_n,\gamma))\cap C((0,\infty);D(L_n ))$ 
and satisfies
\[
\left\{ \begin{array}{l}
g'(t) = L_ng(t), \quad t>0, 
\\
\\
g(0) = 
w_j \circ \pi_F .
\end{array}\right. 
\]
In the language of semigroup theory, $g$ is a classical solution to the above Cauchy problem in the space $L^2(\Omega_n, \gamma)$. 
It is well known that the classical solution is unique; in our case  it coincides with 
$T^{(n)}_t (w_j \circ \pi_F)_{|\Omega_n}= T^{(n)}_t v_{j|\Omega_n}$. Proposition \ref{mainpropFD}  yields
\[
\int_{\Omega_n} |\nabla_H T^{(n)}_t v_{j|\Omega_n}|_H d\gamma=
\int_{\mathscr{O}_n} |\nabla T^F_t w_{j|\mathscr{O}_n}|d\gamma_F (y)
\leq |D_{\gamma_F} w_j|(\mathscr{O}_n) = |D_\gamma v_j|(\Omega_n).
\]
Let us recall that $v_j={\mathbb E}_ju\to u$ in $L^2(X,\gamma)$ and in variation.  
Therefore, taking into account Proposition \ref{mainpropFD} and \eqref{conditineq} we obtain
\begin{align*}
\int_{\Omega_n}|\nabla_H T^{(n)}_t u_{0|\Omega_n}|_H d\gamma  
= &
\lim_{j\to +\infty} \int_{\Omega_n} |\nabla_H T^{(n)}_t v_{j|\Omega_n}|_H d\gamma  
\\
\leq& \liminf_{j\to +\infty} |D_\gamma v_j| (\Omega_n) 
= |D_\gamma u_0| (\Omega_n) \leq |D_\gamma u_0| (\overline\Omega_n).
\end{align*}
Now, as a consequence of Proposition \ref{approxSG} and the hypothesis $|D_\gamma u_0|(\partial \Omega)=0$ 
we obtain  
\begin{align}
\nonumber
\int_\Omega |\nabla_H T_t u_0|_H d\gamma = &
\lim_{n\to +\infty} \int_{\Omega} |\nabla_H T^{(n)}_tu_{0|\Omega_n}|_H d\gamma  
\leq \limsup_{n\to +\infty} \int_{\Omega_n} |\nabla_H T^{(n)}_tu_{0|\Omega_n}|_H d\gamma  
\\
\label{lastEst}
\leq&\lim_{n\to +\infty} |D_\gamma u_0|(\overline{\Omega}_n)
=|D_\gamma u_0|(\Omega) = |D_\gamma u_0|(\Omega), 
\end{align}
which finishes the proof of the Theorem. 
\end{proof}

\begin{Remark}
It is worth noticing that the proof of Theorem \ref{mainThm}, estimate \eqref{lastEst}, yields that also in the infinite 
dimensional setting the map 
\[
t \mapsto \int_\Omega |\nabla_H T_t u_0(x)|_H d\gamma(x) 
\]
is monotone decreasing, for any $u_0\in BV(X,\gamma)\cap L^2(X,\gamma)$.
\end{Remark}

\appendix

\section{Finite dimensional convex analysis}

This section is devoted to  recall some   properties of convex sets and convex functions in Euclidean 
spaces. Most of these results can likely be found in the literature, but we recall 
here some of the proofs for the reader's convenience. Let $C\subset \R^d$ be a closed convex set
with interior part $C^\circ\neq\emptyset$. Possibly translating $C$, without loss of generality we may assume that $0\in C^\circ$.

If $C$ is unbounded, then there exists 
$\nu\in {\mathbb S}^{d-1}$ such that $ t\nu\in C$ for all $t\geq0$. Indeed, if $(x_j)_j\subset C$ is 
a sequence with $\|x_j\|\to +\infty$, then 
\[
\nu_j=\frac{x_j }{\|x_j \|}\subset {\mathbb S}^{d-1}
\]
admits an accumulation point $\nu$; convexity and closedness of $C$ imply that
\[
 t\nu\in C,\qquad \forall t\geq 0.
\]
We set
\[
{\mathbb S}^{d-1}_{C }=\{ \nu\in {\mathbb S}^{d-1}:  t\nu\in C,\quad \forall t\geq 0\}
\]
and we define the maximal cone with vertex at $ 0$ contained in $C$, 
\[
K_C =\{  t\nu: t\geq 0, \nu\in {\mathbb S}^{d-1}_{C }\},
\]
while we set  $K_C =\{ 0\}$ if $C$ is bounded.

We define the map
\[
\m(x)=\inf \{ \lambda\geq 0: x \in \lambda C\}. 
\]
If $x\in K_C $ we have $\m(x)=0$. If $x\not\in K_C $ there exists a unique point $y\in  \partial C$ such that 
\[
x =\m(x) y.
\]
We set 
\[
y= p_C(x). 
\]

\begin{Proposition}\label{propMink}
Let $C\subset \R^d$ be a closed convex set and let $ 0\in C^\circ$. Setting  
\[
r=\sup \{ t>0 : B_t( 0)\subset C\},
\]
$\m$ is convex,  $\frac{1}{r}$--Lipschitz continuous and $C=\{\m\leq1\}$. In addition, if 
$\partial C$ is $C^1$, $\m$ is differentiable at any point $x\not\in \partial K_C $; at such points 
$\langle \nabla \m(x), x\rangle = \m(x)$. 
\end{Proposition}
\begin{proof} First of all we remark that $\m$ is positively homogeneous, namely $ \m(t x) = t \m(x)$ for every $t >0$ and $x\in \R^d$. 

Let us show that  $\m$ is convex. As a first step we show that for any $y_1,y_2\in \R^d$
\[
\m(y_1+y_2)\leq \m(y_1)+ \m(y_2).
\]
Indeed, for all $t_i> \m(y_i)$, $i=1,2$, we have $y_i\in t_i C$. Since 
$C$ is convex, then   $y_1+y_2\in (t_1+t_2) C$, i.e., $ \m(y_1+y_2)\leq t_1+t_2$.

Let now $x_1,x_2\in \R^d$ and $\lambda\in (0,1)$. Using the above inequality
and recalling that $\m$ is homogeneous, we obtain
\[
\m(\lambda x_1+(1-\lambda)x_2)\leq \m(\lambda x_1)+\m((1-\lambda)x_2) 
=\lambda \m(x_1)+(1-\lambda)\m(x_2). 
\]

Let us  show that $\m$ is  Lipschitz continuous. For any $t<r$ and any $x\in \R^d$, 
\[
t\frac{x}{\|x\|}\in B_t(0)\subset C,
\]
that is $x \in \frac{\|x\|}{t} C$, whence in particular $\m(x)\leq \frac{\|x\|}{t}$,
and  letting $t\to r$, we obtain $\m(x)\leq \frac{1}{r} \|x\|$.
As a consequence, for any $x,y\in \R^d$,
\[
\m(x)= \m(y+ x-y) \leq \m(y) +\m(x-y) \leq \m(y)+\frac{1}{r}\|x-y\|, 
\]
which implies 
\[
|\m(x)-\m(y)|\leq \frac{1}{r}\|x-y\|,\qquad \forall x,y \in \R^d.
\]

Let us prove the statements about the regularity of  $\m$. Every $p\in 
 \partial C$ has a neighborhood $U$ such that $\partial C\cap U$ is
the zero level of a $C^1$ function $f$ whose gradient does not vanish at $\partial C$. The function of $(d+1)$ variables 
\[
g(x,\lambda)=f\left(\frac{1}{\lambda}x\right),
\]
is well defined in a neighborhood of $(p, 1)$.  $g$ implicitly defines the Minkowski functional $\m$, since for every $x$ outside $K_C$ and $\lambda >0$,    $x/\lambda\in \partial C$ iff $\lambda = \m(x)$. Moreover, $\partial g(x,\lambda)/\partial \lambda  = -\lambda^{-2}\langle x, \nabla f(x/\lambda)\rangle $ does not vanish at any $(x, 1)$ with $x\in \partial C$, otherwise  the tangent hyperplane at $x$ would contain the origin, which is impossible since $C$ is convex. 

This shows that $\m$ is $C^1$ outside $\overline{K}_C$. Since $\m\equiv 0$ in $K_C$, it follows that $\m$ is $C^1$ outside $\partial K_C$. 
The equality $\langle \nabla \m(x), x\rangle = \m(x)$ at such points  follows from the Euler Theorem on homogeneous functions. 
  \end{proof}

We state the following technical lemma that is used in the proof of Lemma \ref{lemmaConvL1Unif}.

\begin{Lemma}\label{lemmaLowerBdd}
Let $C_n, C\subset \R^d$ be closed convex sets with $C_n$ converging in $L^1_{\rm loc}(\R^d)$ to $C\neq \emptyset$,
that is 
\[
\lim_{n\to +\infty} \mathscr{L}^d((C_n\Delta C)\cap \bar B_R)=0,\qquad
\forall R>0;
\]
then for every $r>0$ there exists $\alpha(r)>0$ such that for any bounded sequence $x_n\in \partial C_n$
\[
\mathscr{L}^d(C_n\cap B_r(x_n))\geq \alpha(r).
\]
\end{Lemma}
\begin{proof}
Let us assume that there exists a sequence $x_n\in \partial C_n\cap B_{\frac{R}{2}}$ with $R>2r$ such that
\[
\mathscr{L}^d(C_n\cap B_r(x_n)) \leq \frac{1}{n};
\]
we define the sets
\[
S_n = \left( \frac{C_n-x_n}{r}\right) \cap \mathbb{S}^{d-1}
\]
and the cone
\[
K_n=\{ x_n+t\nu: \nu\in S_n, t\geq 0\}.
\]
Since $C$ is convex, then
\[
K_n\cap B_r(x_n)\subset C_n\cap B_r(x_n), \quad
C_n\setminus B_r(x_n)\subset K_n\setminus B_r(x_n). 
\]
So in particular we have 
\begin{align*}
\frac{1}{n}\geq &
\mathscr{L}^d(C_n\cap B_r(x_n)) \geq
\mathscr{L}^d(K_n\cap B_r(x_n)) 
=\int_0^r \mathscr{H}^{d-1}(tS_n) dt
=\frac{r^d}{d} \mathscr{H}^{d-1} (S_n),
\end{align*}
and then $\mathscr{H}^{d-1}(S_n)\leq \frac{d}{nr^d}$. On the other hand, if we set
\[
K_{n,R}=\{ x_n+t\nu: \nu\in S_n, 0\leq t\leq 2R\}= K_n\cap B_{2R}(x_n),
\]
we also have   $C_n\cap (\overline{B}_R\setminus B_r(x_n))\subset K_n\cap (\overline{B}_R\setminus B_r(x_n)) 
\subset K_{n,R}\setminus B_r(x_n)$ and then,
since 
\[
\mathscr{L}^d (K_{n,R}\setminus B_r(x_n))=\int_r^{2R} \mathscr{H}^{d-1}(tS_n)dt
=\frac{2^dR^d-r^d}{d}\mathscr{H}^{d-1}(S_n),
\]
we obtain  
\begin{align*}
\mathscr{L}^d (C\cap \bar{B}_R)=&
\lim_{n\to +\infty} \left(
\mathscr{L}^d (C_n\cap \bar B_r(x_n))+\mathscr{L}^d (C_n\cap B_R\setminus \bar B_r(x_n))
\right)\\
\leq &
\lim_{n\to +\infty} \left(
\frac{1}{n}+\frac{2^dR^d-r^d}{d}\mathscr{H}^{d-1}(S_n)
\right)\\
\leq &
\lim_{n\to +\infty} \frac{1}{n}\left(
1+\frac{2^dR^d-r^d}{r^d}
\right)=0,
\end{align*}
and this is a contradiction.
\end{proof}

In the next lemma we show the connection between the $L^1$ convergence of characteristic functions 
of convex sets and the convergence of boundaries.
We recall that the Hausdorff distance between two sets $A,B\subset F$ is defined as
\[
d_\mathscr{H}(A,B)=\inf\{ t: A\subset (B)_t \mbox{ and } B\subset (A)_t \},
\]
where $(A)_t=\{ x\in F: {\rm dist}(x,A)<t\}$. On compact sets this distance induces the Kuratowski 
convergence. A sequence of compact sets $K_j$ converges to a set $K$ in the sense of Kuratowski if 
\begin{enumerate}
\item for any sequence $(x_j)_j$ of elements $x_j\in K_j$, if $x_j \to x$, then $x\in K$;
\item for any $x\in K$, there exists a sequence $(x_j)_j$ of elements $x_j\in K_j$ such that $x_j\to x$.
\end{enumerate}
Indeed, if $x_j\in K_j$ for every $j$ and $x_j\to x$, then $x\in K$ because for every $\varepsilon>0$ 
the points $x_j$ definitively belong to $K_\varepsilon$. Moreover, fixed $x\in K$, for every $j\in\N$ 
there is $\nu_j$ such that $K_i\subset K_{1/j}$ for $i\geq \nu_j$, hence we may select a sequence of points  
$x_j\in K_j$ converging to $x$.

\begin{Lemma}\label{lemmaConvL1Unif}
Let $C\subset \R^d$ be a convex set and let $(C_n)_n \subset \R^d$ be a sequence of convex sets
such that 
\[
\lim_{n\to +\infty} \gamma_F (C\Delta C_n) =0. 
\]
Then $\partial C_n$ converges uniformly on compact sets to $\partial C$, that is for every compact set 
$K$ the sequence $\partial C_n\cap K$ converges to $\partial C\cap K$ in the Hausdorff distance.
\end{Lemma}
\begin{proof}
It suffices to prove the statement for $K=\overline{B_R(0)}$; we have 
\[
{\mathcal L}^d ((C_n\Delta C)\cap B_{R+1}(0))\leq 
(2\pi)^{\frac{d}{2}} e^{\frac{(R+1)^2}{2}} \gamma_F (C_n\Delta C).
\]
Assume by contradiction that there exists $\varepsilon_0>0$ such that for infinitely many $n\in \N$, 
either $\partial C_n\cap \overline{B_R(0)}\not\subset (\partial C)_{\varepsilon_0}\cap \overline{B_R(0)}$, or 
$\partial C\cap \overline{B_R(0)}\not\subset (\partial C_n)_{\varepsilon_0} \cap \overline{B_R(0)}$. In the first case 
there are infinitely many $n\in \N$ for which there exists $x_n\in \partial C_n\cap \overline{B_R(0)}$ but 
$x_n\not\in (\partial C)_{\varepsilon_0}\cap \overline{B_R(0)}$; we have two possibilities, either 
$B_{\varepsilon_0}(x_n)\subset C^\circ$ or $B_{\varepsilon_0}(x_n)\subset \R^d\setminus C$. 
If $B_{\varepsilon_0}(x_n)\subset C^\circ$,  then
\[
{\mathcal L}^d((C_n\Delta C)\cap B_{R+1}(0))\geq \mathcal{L}^d(B_{\varepsilon_0}(x_n)\setminus C_n)
\geq \frac{1}{2}\omega_d \varepsilon_0^d. 
\]
If $B_{\varepsilon_0}(x_n)\subset \R^d\setminus C$, then  by Lemma \ref{lemmaLowerBdd}  
\[
{\mathcal L}^d((C_n\Delta C)\cap B_{R+1}(0))\geq \mathcal{L}^d(B_{\varepsilon_0}(x_n)\cap C_n)
\geq \alpha(\varepsilon_0);
\]
In both cases
\[
\limsup_{n\to +\infty}{\mathcal L}^d((C_n\Delta C)\cap B_{R+1}(0))>0.
\]
Similarly, if there exists $x\in \partial C\cap \overline{B_R(0)}$ such that for infinitely many $n\in \N$,
$B_{\varepsilon_0}(x)\cap \partial C_n=\emptyset$, then either $B_{\varepsilon_0}(x)\subset C_n^\circ$
or $B_{\varepsilon_0}(x)\subset \R^d\setminus C_n$, and then again either
\[
{\mathcal L}^d((C_n\Delta C)\cap B_{R+1}(0))\geq \mathcal{L}^d(B_{\varepsilon_0}(x)\setminus C)
\geq \frac{1}{2}\omega_d \varepsilon_0^d,
\]
or
\[
{\mathcal L}^d((C_n\Delta C)\cap B_{R+1}(0))\geq \mathcal{L}^d(B_{\varepsilon_0}(x)\cap C), 
\]
so that, again, 
\[
\limsup_{n\to +\infty}{\mathcal L}^d((C_n\Delta C)\cap B_{R+1}(0))>0.
\]
contradicting the fact that ${\mathcal L}^d((C_n\Delta C)\cap B_{R+1}(0))\to 0$ as $n\to +\infty$.
\end{proof}

\begin{Proposition}\label{propApproxConv}
Let $C\subset \R^d$ be a closed convex set. Then  for any $\delta>0$ there exists a closed convex set 
$C_\delta$ such that $C\subset C_\delta^{\circ}$, $\partial C_\delta$ is smooth and
\[
\lim_{\delta \to 0} \gamma_F\left( C_\delta\setminus C\right)=0. 
\]
\end{Proposition}
\begin{proof} Fix $\delta >0$ and set 
\[
(C)_\delta =\{ y: d(y,C)\leq\delta\}. 
\]
Then $(C)_\delta$ is convex and contains $C$. A result of Federer \cite{Fed59Cur} implies that the 
boundary of $(C)_\delta$ is $C^{1,1}$ if  $\delta$ is sufficiently small; however that is not enough 
for our aims. 

Fix $\delta >0$, let $\m$ be the Minkowski function of  $(C)_\delta$ and let 
$\varrho \in C^{\infty}_c(\R^d)$ be a standard mollifier. For $\eta >0$ define
as usual $\varrho_\eta(x) = \varrho(x/\eta)/\eta^d$,  
\[
\m_\eta  =\m*\varrho_\eta, 
\]
and consider the set 
\[
C_{\delta} = \{ x\in \R^d:\; \m_{\delta }  (x)\leq 1\}. 
\]
Since $\m$ is convex, each $\m_\eta$ is convex too. Indeed,   
\begin{align*}
\m_\eta (\lambda x_1+(1-\lambda)x_2)
=& \int \m(\lambda x_1+(1-\lambda) x_2-y)\varrho_\eta(y)dy \\
=&\int \m(\lambda (x_1-y)+(1-\lambda) (x_2-y))\varrho_\eta(y)dy \\
\leq &\lambda \int \m(x_1-y)\varrho(y)dy +(1-\lambda)\int \m(x_2-y)\varrho(y)dy \\
=&\lambda\m_\eta(x_1)+(1-\lambda)\m_\eta(x_2) .
\end{align*}
Therefore,  $C_{\delta}$ is a convex set. 

Let us prove that $C_{\delta}^{\circ} \supset C$. For every $x\in C$, the ball $B_{\delta}(x)$ is 
contained in $(C)_{\delta}$, and then 
\[
\m_\delta(x)=\int_{B_\delta(x)} \m(y)\varrho_\delta(x-y)dy \leq 1
\] 
since $\m(y)\leq 1$ for all $y\in B_\delta(x)$. This shows that $C \subset C_{\delta}$. To prove 
the inclusion $C \subset C_{\delta}^{\circ}$ we remark that $\overline{B_{\delta}(x)}\cap C$ has 
positive Lebesgue measure and it is contained in $(C)_{\delta}^{\circ}$, therefore the restriction 
of  $\m$ to $\overline{B_{\delta}(x)}\cap C$ has maximum strictly less than $1$ and the
integral above is strictly less than $1$. This shows that $C \subset C_{\delta}^{\circ}$. 

Let us  prove that if $\delta $ is sufficiently small then  the boundary of $C_{\delta} $ is smooth. 
We have only to show that the gradient of $\m_\delta$ does not vanish at the boundary. To this aim it 
is sufficient to show that  for every $x$ such that $ \m_{\delta}(x)=1$ we have 
$\langle \nabla \m_{\delta}(x), x\rangle \neq 0$. 

Let $r>0$ be such that $B_r(0)\subset C$ and let 
\[
(i) \;\delta <r/4, \quad (ii)\; \delta \int_{\R^d} |u|\varrho(u)du < r/2. 
\]
For every $x$ such that $ \m_{\delta}(x)=1$ there exists $\bar x\in B_\delta(x)$ such that 
$\m(\bar x)\geq 1$ (otherwise we would get $ \m_{\delta}(x)<1$).  Since 
$B_r(0)\subset ( C)_{\delta}$, then $\m$ is $1/r$-Lipschitz, so that  for every $y\in \R^d$ we have
$\m(y) \geq \m({\bar x}) - \|y - {\bar x}\|/r$, and hence by (i)
\[
\m(y)\geq \frac{1}{2},\qquad \forall y\in B_\delta(x). 
\]
Consequently, 
\begin{align*}
\langle \nabla \m_{\delta}(x), x\rangle & = 
\int_{\R^d}  \langle \nabla \m (y), x-y\rangle  \varrho_{\delta}(x-y)dy + 
\int_{\R^d}  \langle \nabla \m (y), y\rangle  \varrho_{\delta}(x-y)dy
\\
&= \int_{\R^d}  \langle \nabla \m (y), x-y\rangle  \varrho_{\delta}(x-y)dy + 
\int_{B_{\delta}(x)}   \m (y)   \varrho_{\delta}(x-y)dy.
\end{align*}
The modulus of the first integral does not exceed
\[ 
\frac{1}{r} \int_{\R^d} |x-y| \varrho_{\delta}(x-y)dy   = 
\frac{\delta}{r} \int_{\R^d} |u|\varrho(u)du <\frac{1}{2}
\]
while the second integral is $\geq 1/2$. Therefore,   $\langle \nabla \m_{\delta}(x), x\rangle >0$. 

To prove the last statement it is sufficient to show that for every $x\notin C$ we have 
$\m_{\delta}(x)>1$ if $\delta $ is small enough. Indeed, in this case $\one_{C_{\delta}\setminus C}$ 
goes to $0$ pointwise as $\delta \to 0$, so that $\gamma_F(C_{\delta}\setminus C)$ vanishes as 
$\delta \to 0$. Let $\delta_0= $ dist$(x, C)>0$ and let $\delta<\delta_0/2$. Then 
$\overline{B_{\delta}(x)} \cap (C)_\delta =\emptyset $, so that $\m >1$ in 
$\overline{B_{\delta}(x)} \cap (C)_\delta $ and 
$\min\{\m(y):\;y\in \overline{B_{\delta}(x)}  \cap (C)_\delta \}>1$. Consequently, 
\[
\m_{\delta}(x) = \int_{B_x(\delta)} \m(y)\varrho_{\delta}(x-y)dy >1. 
\]
\end{proof}

We conclude this section approximating  an infinite dimensional open 
convex set by finite dimensional regular open convex sets. 

\begin{Proposition}
Let $\Omega\subset X$ be an open convex set. Then there exists a sequence of open convex cylindrical 
sets $\Omega_n\supset \Omega_{n+1}\supset \Omega$ with   smooth boundaries, such that 
  \[
\lim_{n\to +\infty} \gamma (\Omega_n\setminus \Omega)=0.
\]
\end{Proposition}
\begin{proof}
Since $\Omega$ is an open convex set, then $\gamma(\partial \Omega)=0$. Since $\overline\Omega$ is a closed 
convex set and $X$ is separable, by the  Lindel\"of theorem, see e.g. \cite[Theorem I.4.14]{DS} we have 
\[
\overline\Omega =\bigcap_{j\in \N} \overline S_j,
\]
where
\[
S_j=S(x^*_j, a_j)=\{ x\in X: x^*_j(x) < a_j\}
\] 
with $x^*_j\in X^*\setminus\{0\}$, $a_j\in \R$, are open half-spaces containing $\Omega$. The set
\[
A_n=S_1\cap\ldots\cap S_n
\]
is an open convex set containing $\Omega$,  and  
$\overline A_n =\overline S_1\cap\ldots\cap\overline S_n$ contains $\overline\Omega$. Then, 	
\[
\gamma(\Omega)=\gamma(\overline \Omega) = \lim_{n\to +\infty} \gamma(\overline A_n).
\]
We denote by $F_n$ the linear span of the vectors $x^*_1,\ldots, x^*_n$, which is a subspace of $H$ 
of dimension $d\leq n$. We fix an orthonormal (along $H$) basis $\{h_1, \ldots, h_d\}$ of $F_n$ 
contained in $Q(X^*)$ and we define the projection 
$\Pi_n :X\to F_n$,  $\Pi_n (x) = \sum_{j=1}^d \hat{h}_j(x)h_j$. 
The induced measure $ \gamma \circ \Pi_n^{-1}$ in $F_n$ is denoted by $ \gamma_n$; if $F_n$ is 
identified with $\R^d$ through the isomorphism $h\mapsto ([h, h_1]_H, \ldots [h, h_d]_H)$, then 
$ \gamma_n$ is just the standard Gaussian measure in $\R^d$. 

Then, $\overline A_n=\Pi_n^{-1}(C_n)$, $C_n$ is a polyhedral closed convex set in $F_n$ with 
$\gamma(\overline A_n)=\gamma_n (C_n)$. By Proposition \ref{propApproxConv}, for any $n$ we find a 
smooth open convex set $\mathscr{O}_n$ with smooth boundary such that $C_n\subset \mathscr{O}_n$ and
\[
\gamma_n (\mathscr{O}_n\setminus C_n)\leq \frac{1}{n}.
\] 
We may then define $\Omega_n=\Pi_n^{-1}(\mathscr{O}_n)$. Such sets are open cylindrical convex sets,
and 
$\Omega\subset \Omega_n$ for any $n\in \N$. If $F_n=F_{n+1}$, i.e. $x^*_{n+1}$
is a linear combination of $x^*_1,\ldots,x^*_n$, then $C_{n+1}\subset C_n$, otherwise $F_{n+1}=F_n\times \R$ 
and $C_{n+1}\subset C_n\times \R$.   To get  the inclusion $\Omega_{n+1}\subset\Omega_n$, 
it suffices to apply Proposition \ref{propApproxConv} with decreasing sequences $(\delta_n)$  
in place of $\delta $. 

Moreover, 
\begin{align*}
\gamma(\Omega) \leq &
\liminf_{n\to+\infty} \gamma (\Omega_n)
=
\liminf_{n\to +\infty} \gamma_n (\mathscr{O}_n) 
\leq 
\lim_{n\to +\infty} \left( 
\gamma_n (C_n)+\frac{1}{n}
\right)
=\lim_{n\to +\infty} \gamma(\overline A_n) =\gamma(\overline \Omega)
=\gamma(\Omega),
\end{align*}
then the conclusion follows. 
\end{proof}

\end{document}